\documentclass{birkjour}
\usepackage{mathtools}
\usepackage{amssymb, amsmath, amsfonts, amsthm}
 \newtheorem{thm}{Theorem}[section]
 \newtheorem{cor}[thm]{Corollary}
 \newtheorem{lem}[thm]{Lemma}
 \newtheorem{prop}[thm]{Proposition}
 \theoremstyle{definition}
 \newtheorem{defn}[thm]{Definition}
 \theoremstyle{remark}
 
 \newtheorem*{ex}{Example}
 \numberwithin{equation}{section}

\begin{document}

%
%
%
%
%
%
%
%
%

\title[Double bases from generalized Faber polynomials]
 {Double bases from generalized Faber polynomials  with complex-valued coefficients in weighted Lebesgue spaces}

\author[B.T. Bilalov]{B.T. Bilalov}

\address{%
Department of "Non-harmonic Analysis"\\
Institute of Mathematics and Mechanics of NAS of Azerbaijan\\
B.Vahabzade 9, AZ1141\\
Baku, Azerbaijan}

\email{b\_bilalov@mail.ru}

\thanks{This research was supported by the Azerbaijan National Academy of Sciences under the program "Approximation by neural networks and some problems of frames", 2017.}
\author{A.A. Huseynli}
\address{$^a$Department of Mathematics, Khazar University, AZ1096, Baku, Azerbaijan\br
$^b$Department of Non-harmonic analysis", Institute of Mathematics and Mechanics of NAS of Azerbaijan\br
AZ1141, Baku, Azerbaijan}
\email{alihuseynli@gmail.com}
\author{S.R. Sadigova}
\address{Department of "Non-harmonic Analysis"\br
Institute of Mathematics and Mechanics of NAS of Azerbaijan\br
B.Vahabzade 9, AZ1141\br
Baku, Azerbaijan}
\email{s\_sadigova@mail.ru}


\date{October 2, 2018}

\begin{abstract}
In the paper it is considered the generalized Faber polynomials defined inside and outside a regular curve on the complex plane. The weighted Smirnov spaces corresponding to bounded and unbounded regions are defined. It is proved that the generalized Faber polynomials forms a basis in weighted Smirnov spaces, if the weight function satisfies the Muckenhoupt condition on the regular curve. The double system of generalized Faber polynomials with complex-valued coefficients is also considered and the basis properties of such a system in weighted Lebesgue spaces over regular curves are studied.
\end{abstract}

\maketitle
\section{Introduction}

The Faber polynomials were introduced in 1903 by G. Faber in connection with applications of approximation on complex plane. The detailed information about these problems one may consult the monograph by V. I. Smirnov and N.A. Lebedev \cite{1}, also D. Gayer \cite{2}.  They replace polynomials of the variable $z$  in a circle with respect to simply-connected domains. These polynomials play an important role in the problems of approximation on the complex plane and in the theory of conformal mappings. Series of classical Faber polynomials have been investigated enough well, and the results obtained here were comprehensively covered in \cite{1}. The $L_{p} $  case of the classical results was initiated to be studied by V. Kokilashvili in \cite{3}, where generalized Faber polynomials were introduced. The connection between generalized Faber and classical Faber polynomials was studied by J.E. Andersson in \cite{4}. For some other aspects of approximation by Faber polynomials, see \cite{2} and \cite{40}. Note that the basisness problem of the system of generalized Faber polynomials in the Lebesgue spaces of functions defined on  rectifiable closed Jordan curves was studied by B.T. Bilalov and T.I. Najafov in \cite{5}. The degree of the approximation by generalized Faber polynomials in Smirnov's spaces was investigated by \cite{6}.

 To study the basisness of generalized Faber polynomials we will use the method of Riemann-Hilbert boundary value problems for analytic functions. This idea takes its origins from the note \cite{7}  of A. V. Bitsadze. The method was successfully used by S. Ponomarev \cite{8,9} and E. I. Moiseev \cite{10,11,12,13,14,15,16} to solve mixed type PDEs on special regions by the method of Fourier, and also to prove the criteria which guarantees the basisness of a trigonometric system with linear phase in Lebesgue spaces. The further development of this method to study the basisness (completeness, minimality and basisness) problems of special system of functions can be found in the work of B.T.Bilalov \cite{17,18,19,20,21,22}. This method is still intensively developing, the boundary value problems and basisness problems are studied in various function spaces (see, e.g. \cite{23,24,25,26,27,28,29,30}). To establish the basisness of the double system of generalized Faber polynomials in weighted Lebesgue spaces we heavily use the Riemann boundary  value problems posed in weighted Smirnov spaces. These problems were studied in \cite{31}. Note that similar problems in various formulation were studied in \cite{32,33,34,35,36}.

The present  paper considers generalized Faber polynomials defined inside and outside a regular curve on the complex plane. The weighted Smirnov spaces corresponding to bounded and unbounded regions are defined. It is proved that the generalized Faber polynomials forms a basis in weighted Smirnov spaces, if the weight function satisfies the  Muckenhoupt condition on the regular curve. It is also considered the double system of generalized Faber polynomials with complex-valued coefficients and the basisness of these systems in weighted Lebesgue spaces over regular curves are studied. It should be noted that in the weightless case these problems were previously studied in \cite{5} and here we will follow the scheme of this work.

\section{Preliminaries}

 We give general notations and some definitions from the approximation theory and the theory of singular integral operators, which we will use through. By $O_{r} \left(z\right)$ we denote a disc with the radius $r$ and the center $z_{0} $ in the complex plane: $O_{r} \left(z_{0} \right)\equiv \left\{z\in \mathbb C:\, \left|z-z_{0} \right|<r\, \right\}$ ($\mathbb C$ is the complex plane); $\left|M\right|$ denotes the arc measure of  the set $M\subset \Gamma $, where $\Gamma \subset \mathbb C$ is a rectifiable curve. Denote $\omega =O_{1} \left(0\right)$ and $T=\partial \omega $.  We will also use the following standard  notation.  $\mathbb Z_{+} =\left\{0\right\}\cup \mathbb N$; $\mathbb Z$ is the set of all integers. 

\begin{defn}\label{D2.1} A  rectifiable Jordan curve $\Gamma $ in the complex plane is called Carleson or regular if 
\[{\mathop{\sup }\limits_{z\in \Gamma }} \left|\Gamma \cap O_{r} \left(z\right)\right|\le cr\, \, ,\, \, \, \forall r>0\, ,\] 
where $c$ is a constant, independent of $r$.
\end{defn}
More details on this concept can be found in \cite{6,31,32}. 

 Let $ \Gamma$ be a rectifiable Jordan curve and $\rho \left(\cdot \right)$ be a positive function defined a.e. on $ \Gamma$.

\begin{defn}\label{D2.2} The function $\rho \left(\cdot \right)$ is said to belong to the Muckenhoupt class $A_{p} \left(\Gamma \right)$ ($p>1$) on the curve $\Gamma $ if 
\[{\mathop{\sup }\limits_{z\in \Gamma }} {\mathop{\sup }\limits_{r>0}} \left(\frac{1}{r} \int _{\Gamma \cap O_{r} \left(z\right)}\rho \left(\xi \right)\, \left|d\xi \right| \right)\, \left(\frac{1}{r} \int _{\Gamma \cap O_{r} \left(z\right)}\left|\rho \left(\xi \right)\right|^{-\frac{1}{p-1} } \, \left|d\xi \right| \right)^{p-1} <+\infty .\] 
\end{defn}
Let us give the definition of the generalized $p$-Faber polynomials   $F_{p,n}^{+} $ and $F_{p,n}^{-} $  (see \cite{3,6}). Let $D^{+} $ be a bounded region with the boundary $\Gamma $ and  the simple-connected complement $D^{-} =C\backslash \bar{D}^{+} $ ($\bar{D}^{+} $ is the closure of $D^{+} $). Let $w=\varphi \left(z\right)$ be a single-valued conformal mapping the region $D^{-} $ on $C\backslash \overline{O_{1} \left(0\right)}\equiv O_{1}^{-} \left(0\right):\, \varphi \left(\infty \right)=\infty \, ,\, \, \varphi '\left(\infty \right)=\gamma >0$. $\varphi \left(z\right)$ is the sum of its Laurent series at $z=\infty $:
\[\varphi \left(z\right)=\gamma \, z+\gamma _{0} +\gamma _{1} z^{-1} +...\, \, .\] 
Let us take the analytic branch of $\sqrt[{p}]{\varphi '\left(z\right)} $ for which $\sqrt[{p}]{\varphi '\left(\infty \right)} >0$. By $F_{p,n}^{+} $ we denote the principal part of the Laurent series of $\left[\varphi \left(z\right)\right]^{n} \, \sqrt[{p}]{\varphi '\left(z\right)} $ at $z=\infty $:
\[\left[\varphi \left(z\right)\right]^{n} \, \sqrt[{p}]{\varphi '\left(z\right)} \equiv F_{p,n}^{+} \left(z\right)+E_{p,n}^{+} \left(z\right),\, \, z\in D^{-} ,\] 
where $E_{p,n}^{+} \left(\infty \right)=0$.  Here we take $F_{p,0}^{+} \equiv 1$.

Similarly, the  $F_{p,n}^{-} $ $p$-Faber polynomial  corresponding to the mapping $\psi \left(\cdot \right)$ is defined.  Now, let $D^{+} $ be a bounded simply-connected region, containing $z=0$ and  $w=\psi \left(z\right)$ conformal and single-valued function mapping $D^{+} $ on $O_{1}^{-} \left(0\right)\, :\, \psi \left(0\right)=\infty $, ${\mathop{\lim }\limits_{z\to 0}} z\psi \left(z\right)=\alpha >0$. The function $\psi \left(z\right)$ has a Laurent expansion at $z=0$:
\[\psi \left(z\right)=\alpha \, z^{-1} +\alpha _{0} +\alpha _{1} z+...\, \, .\] 
It is clear that the point $z=0$ is a pole of $\left[\psi \left(z\right)\right]^{n-\frac{2}{p} } $ $\sqrt[{p}]{\psi '\left(z\right)} $ of order $n$ and thus 
\[\left[\psi \left(z\right)\right]^{n-\frac{2}{p} } \sqrt[{p}]{\psi '\left(z\right)} =F_{p,n}^{-} \left(z^{-1} \right)+E_{p,n}^{-} \left(z\right),\] 
where $F_{p,n}^{-} \left(z^{-1} \right)$ is the principal part of the series

\[F_{p,n}^{-} \left(z^{-1} \right)=\alpha _{n}^{\left(n\right)} z^{-n} +\alpha _{n-1}^{\left(n\right)} z^{-n+1} +...+\alpha _{1}^{\left(n\right)} z^{-1} \, \, .\]

In what follows, we denote by 
 $\varphi_{-1}:O_{1}^{-}\left(0\right)\rightarrow D^{-}$  and   $\psi_{-1}:O_{1}^{-}\left(0\right)\rightarrow D^{+}$,  the functions inverse to  $\varphi \left(\cdot\right)$ and $\psi \left(\cdot\right)$, respectively. 

 $L_{p,\rho} \left(\Gamma \right)$  is the usual weighted Lebesgue space equipped with the norm $\left\| \cdot \right\| _{p,\rho } $:
\[\left\| f\right\| _{L_{p, \rho} \left(\Gamma \right)} =\left(\int _{\Gamma}\left|f\left(\xi \right)\right|^{p} \rho \left(\xi \right)\left|d\xi \right| \right)^{\frac{1}{p} } .\] 

Consider the Cauchy singular integral operator $S_{\Gamma} $:
\[S_{\Gamma} \left(f\right)=\frac{1}{2\pi i} \int _{\Gamma}\frac{f\left(\xi \right)}{\xi -\tau }  d\xi \, ,\, \, \tau \in \Gamma.\] 

The following theorem is valid.

\begin{thm}\label{D}   $S_{\Gamma} $ is bounded in $L_{p} \left(\Gamma\right)$, $1<p<+\infty $, if and only if $\Gamma $ is a regular curve. Furthermore, if $ \Gamma$ is a regular curve then $S_{\Gamma} $ is bounded in $L_{p, \rho} \left(\Gamma \right)$, $1<p<+\infty $, if and only if $\rho \in A_{p} \left(\Gamma\right)$.
\end{thm}

For these and related results see, for example, \cite{41,42,43}.

Many facts and concepts we need are given in \cite{5}.  Let us give this information for the sake of ease of reading. We will use some facts on basisness of the classical system of exponents and its part in the weighted Lebesgue and Hardy spaces. Recall the definition of Hardy classes and its weighted counterpart. 

The definitions of the classical Hardy classes $H_{p}^{+} $ and $_{m}H_{p}^{-} $ of analytic functions  inside and outside the unit circle, respectively, can be found in  \cite{5}. Now we define the weighted counterparts of the Hardy classes. Let
\[\tilde{H}^{+} \equiv \left\{f\in H_{1}^{+} :\, f^{+} \in L_{p,\, \nu ^{+} } \right\},\] 
where $H_{1}^{\pm } $ are Hardy classes of functions defined inside and outside of the unit disc, respectively, $\nu ^{+} \left(\cdot \right)$ is a weight function defined on $\left[-\pi ,\pi \right]$, $L_{p,\, \nu ^{+} } $ is a weighted Lebesgue space on $\left(-\pi ,\, \pi \right)$, $f^{+} \left(e^{it} \right)$ is the non-tangential boundary value of $f\in H_{1}^{+} $. Equip $\tilde{H}^{+} $ with the following norm
\begin{equation} \label{GrindEQ__2_1_} 
\left\| f\right\| _{\tilde{H}^{+} } \equiv \left\| f^{+} \left(e^{it} \right)\right\| _{p,\, \nu ^{+} },              
\end{equation} 
where $\left\| \, \cdot \, \right\| _{p,\, \nu ^{+} } $ is the norm of $L_{p,\, \nu ^{+} } $:
\[\left\| f\right\| _{p,\nu ^{+} } =\left(\int _{-\pi }^{\pi }\left|f(t\right|^{p} \nu ^{+} \left(t\right)dt \right)^{\frac{1}{p} },\]
and the corresponding space denote by $H_{p,\nu}^{+}.$ 

It is easy to prove the following 

\begin{prop}\label{P2.1} Let $\left|\nu ^{+} \right|^{-\frac{p}{q} } \in L_{1} ,$ $\frac{1}{p} +\frac{1}{q} =1,\, $ $1\le q<+\infty $. Then the space $H_{q,\, \nu ^{+} }^{+} $  is a Banach space.
\end{prop}

 Let ${}_{m} H_{1}^{-} $ be the Hardy class of  the functions, which are analytic outside the unit disc and has a zero of order not greater than $m$ at infinity. Let
\[{}_{m}H_{q,\, \nu ^{-} }^{-} \equiv \left\{f\in _{m} H_{1}^{-} :\, f^{-} \left(e^{it} \right)\in L_{q,\, \nu ^{-} } \right\},\] 
where $\nu ^{-} $ is a weighted function on $\left[-\pi ,\pi \right]$. The following proposition is the analog of  the above one. 

\begin{prop}\label{P2.2} If $\left|\nu ^{-} \right|^{-\frac{p}{q} } \in L_{1} $, $\frac{1}{p} +\frac{1}{q} =1$, $1\le q<+\infty $ then ${}_{m}H_{q,\, \nu ^{-} }^{-}$ is a Banach space with respect to the norm
\[\left\| f\right\| _{{}_{m}H_{q,\, \nu ^{-} }^{-} } \equiv \left\| f^{+} \left(e^{it} \right)\right\| _{q,\, \nu ^{-} } .\] 
\end{prop}

Throughout this paper,  $q$ will denote the conjugate of a  number   $p$, i.e. $\frac{1}{p}+\frac{1}{q}=1$.
The restrictions of the functions belonging to $H_{p,\, \nu }^{+} $ and ${}_{m} H_{p,\, \nu }^{-} $ to the unit circle is denoted by $L_{p,\, \nu }^{+} $ and ${}_{m} L_{p,\, \nu }^{-} $, respectively: ${H_{p,\, \nu }^{+}  \mathord{\left/{\vphantom{H_{p,\, \nu }^{+}  {}_{\partial \omega } }}\right.\kern-\nulldelimiterspace} {}_{\partial \omega } } =L_{p,\, \nu }^{+} $; ${}_{m} {H_{p,\, \nu }^{-}  \mathord{\left/{\vphantom{H_{p,\, \nu }^{-}  {}_{\partial \omega } }}\right.\kern-\nulldelimiterspace} {}_{\partial \omega } } =_{m} L_{p,\, \nu }^{-} $.

We will say that the weight $\nu \left(\cdot \right)$ defined on $\left[-\pi ,\pi \right]$ belongs to the Muckenhoupt class $A_{p} ,1<p<+\infty $, if
\[{\mathop{\sup }\limits_{I\subset \left[-\pi ;\pi \right]}} \left(\frac{1}{\left|I\right|} \int _{I}\nu \left(t\right)dt \right)\left(\frac{1}{\left|I\right|} \int _{I}\left|\nu \left(t\right)\right|^{-\frac{1}{p -1} } dt \right)^{p -1} <+\infty ,\] 
where $\sup $ takes over all subintervals $I\subset \left[-\pi ,\pi \right],\, \, \, \, \, \left|I\right|$ is the Lebesgue measure of the interval $I$. 

Summarizing the results obtained earlier in \cite{44}, we reach the following

\begin{thm}\label{M-H} The system of exponentials $\left\{e^{int} \right\}_{n\in \mathbb Z} $ forms a basis in $L_{p;\nu } \left(-\pi ,\pi \right)\, $ if and only if $\nu \in A_{p} ,\, \, \, \, 1<p<+\infty $.\end{thm}

By using this theorem the following theorem is proved.

\begin{thm}\label{T2.1}  Let $\nu \in A_{p} ,\, \, \, \, \, \, 1<p<+\infty $. Then: i) the system $\left\{z^{n} \right\}_{n\in \mathbb Z_{+} } $ (i.e. $\left\{e^{int} \right\}_{n\in \mathbb Z_{+} } $) forms a basis in $H_{\rho ,\nu }^{+} $ (i.e.in $L_{\rho ,\nu }^{+} $); ii) the system $\left\{z^{-n} \right\}_{n\ge m} $ (i.e. $\left\{e^{-int} \right\}_{n\ge m} $) forms a basis in ${}_{m} H_{\rho ,\nu }^{-} $ (i.e. in ${}_{m} L_{\rho ,\nu }^{-} $).
\end{thm}

For more comprehensive information about this and related results see, for example, \cite{44}.

At the end of this section let us give the definition of double bases in Banach space. 

\begin{defn}\label{D2.3} The double system $\left\{x^{+}_{n}; x^{-}_{n}\right\}_{n \in \mathbb N} \subset X$ in Banach space $X$ is said to be a bases for  $X$ if $\forall x \in X $ has the unique representation 
\[ x=\sum^{\infty}_{n=1}\lambda_{n}^{+} x_{n}^{+} + \sum^{\infty}_{n=1}\lambda_{n}^{-} x_{n}^{-}, \]
where $\left\{\lambda_{n}^{+}; \lambda_{n}^{-}\right\}_{n \in \mathbb N} \subset \mathbb C$ are scalars. 

\end{defn}

\section{The main assumptions.  Weighted Smirnov classes $E_{p,\rho } \left(D^{+} \right)$   and ${}_{m} E_{p,\rho } \left(D^{-} \right)$}

To study the basis properties of the system of generalized Faber polynomials in weighted Lebesgue spaces we will use the methods of Riemann boundary value problems for analytic functions. We will consider  the Riemann problem in weighted Smirnov classes. In this paragraph we define these weighted classes and prove that under some conditions these spaces are Banach spaces.

Let  $A\left(\xi \right)\equiv \left|A\left(\xi \right)\right|e^{i\alpha \left(\xi \right)} $, $B\left(\xi \right)\equiv \left|B\left(\xi \right)\right|e^{i\beta \left(\xi \right)} $ be complex-valued functions, defined on the curve $\Gamma $. We will assume that the following conditions hold:

i)  $\left|A\right|^{\pm 1} \, ,\, \left|B\right|^{\pm 1} \in L_{\infty } \left(\Gamma \right)$ ;

ii) \textit{$\alpha \left(\xi \right)\, ,\, \, \beta \left(\xi \right)$ are piecewise continuous functions on $\Gamma $ and let\linebreak  $\left\{\xi _{k},k=\overline{1,r}\, \right\}\subset \Gamma $ be the discontinuous points of $\theta \left(\xi \right)\equiv \beta \left(\xi \right)-\alpha \left(\xi \right)$.}

Relative with the curve $ \Gamma$ we assume that the following conditions hold:

iii) \textit{Either $\Gamma $ is a piecewise Lyapunov or Radon curve (i.e. is a curve with bounded rotation) with no cusps. As the direction over $ \Gamma$ we accept the positive direction, i.e. the direction when moving in it the area $D$ remains on the left. Let $a\in \Gamma$ be the initial (also the end) point of the curve $ \Gamma$. We will assume that $\xi \in \Gamma $ follows the point $\tau \in \Gamma $, i.e. $\tau \prec \xi $, if $\xi $ comes after $\tau $ when moving in the positive direction on $\Gamma \backslash a$, where $a\in \Gamma$ is a two stuck points $a^{+} =a^{-} \, \, ,\, \, \, a^{+} -$the start, $a^{-} \, -$end of the curve $\Gamma$.}

\textit{ Denote the class of curves satisfying iii) by ${\rm LR}$.}

Thus, without loss of generality, we will assume that $a^{+} \prec \xi _{1} \prec ...\prec \xi _{r} \prec b=a^{-} $. We denote the one sided limit $  \lim_{\substack{\xi \to \xi _{0} \pm 0\\ {\xi \in \Gamma}}}g\left(\xi \right) $ of the function $g\left(\xi \right)$ at the point $\xi _{0} \in \Gamma $ generated by this order by $g\left(\xi _{0} \pm 0\right)$. Let $h_{k} $:$h_{k} =\theta \left(\xi _{k} +0\right)-\theta \left(\xi _{k} -0\right)$, $k=\overline{1,r}$ be the jumps of the function $\theta \left(\xi \right)$ at points $\xi _{k} \, ,\, k=\overline{1,r}$.

Let us give the definition of Smirnov classes. Let $D^{+} \subset \mathbb C-$be a bounded region with the boundary $\Gamma =\partial D^{+} $ which satisfies iii).  By $E_{p} \left(D^{+} \right)\, ,\, 1<p<\infty $, we denote the Smirnov space of analytic functions on $D^{+} $, which is also a Banach space with respect to the norm $\left\| \, \cdot \, \right\| _{E_{p} \left(D^{+} \right)} $:
\begin{equation} \label{GrindEQ__3_1_} 
\left\| f\right\| _{E_{p} \left(D^{+} \right)} =\, :\, \left\| f^{+} \right\| _{L_{p} \left(\Gamma \right)} \, ,\, \forall f\in E_{p} \left(D^{+} \right),         
\end{equation} 
where $f^{+} ={f \mathord{\left/{\vphantom{f {}_{\Gamma } }}\right.\kern-\nulldelimiterspace} {}_{\Gamma } } -$ is the non-tangential boundary values of $f$ along $\Gamma $. Similarly, the Smirnov space $E_{p} \left(D^{-} \right)$ of functions defined on the region $D^{-} $ with the boundary $\Gamma =\partial D^{-} $ is defined with by norm
\[\left\| f\right\| _{E_{p} \left(D^{-} \right)} =\, :\, \left\| f^{-} \right\| _{L_{p} \left(\Gamma \right)} \, ,\, \forall f\in E_{p} \left(D^{-} \right),\] 
where $f^{-} ={f \mathord{\left/{\vphantom{f {}_{\Gamma } }}\right.\kern-\nulldelimiterspace} {}_{\Gamma } } -$ is the non-tangential boundary values of $f$ along $\Gamma $.

Let $\rho \in L_{1} \left(\Gamma \right)-$ is a weight function. Denote 
\[E_{p,\rho } \left(D^{+} \right)\equiv \left\{f\in E_{1} \left(D^{+} \right):\, \left\| f^{+} \right\| _{L_{p,\rho } \left(\Gamma \right)} <+\infty \right\} ,\] 
and equip it with the norm
\begin{equation} \label{GrindEQ__3_2_} 
\left\| f\right\| _{E_{p,\rho } \left(D^{+} \right)} =\left\| f^{+} \right\| _{L_{p,\rho } \left(\Gamma \right)} .        
\end{equation} 

The Smirnov classes on unbounded regions are defined similarly. Let $D^{-} \subset \mathbb C$ is a unbounded region containing infinity $\left(\infty \right)$. Denote by ${}_{m} E_{1} \left(D^{-} \right)$ the class of functions of $E_{1} \left(D^{-} \right)$, which have the Laurent expansion at $z=\infty $ of the form $f\left(z\right)=\sum _{k=-\infty }^{m}a_{k} z^{k}  $, where $m$ is an integer.

For the weight function $\rho \in L_{1} \left(\Gamma \right)$, the weighted class ${}_{m} E_{p,\rho } \left(D^{-} \right)$ is defined as
\[{}_{m} E_{p,\rho } \left(D^{-} \right)\equiv \left\{f\in _{m} E_{1} \left(D^{-} \right):\, \left\| f^{-} \right\| _{L_{p,\rho } \left(\Gamma \right)} <+\infty \right\}, \] 
and here the norm is given by
\[\left\| f\right\| _{{}_{m} E_{p,\rho } \left(D^{-} \right)} =\left\| f^{-} \right\| _{L_{p,\rho } \left(\Gamma \right)} , \] 
where $f^{-} ={f \mathord{\left/{\vphantom{f {}_{\Gamma } }}\right.\kern-\nulldelimiterspace} {}_{\Gamma } } -$ is the non-tangential boundary values of $f$ along $\Gamma $.

 We have the following

\begin{lem}\label{L3.1} If $\rho ^{-\frac{q}{p} } \in L_{1} \left(\Gamma \right)$ then $E_{p,\rho } \left(D^{+} \right)$, $1<p<+\infty $,  is a Banach space. 
\end{lem}

\begin{proof} Let $\left\{f_{n} \right\}_{n\in \mathbb N} \subset E_{p,\rho } \left(D^{+} \right)$ be any fundamental sequence, that is 

   $\left\| f_{n} -f_{m} \right\| _{E_{p,\rho } \left(D^{+} \right)} \to 0$, as  $n,m\to \infty $.

\noindent Thus

 $\left\| f_{n}^{+} \left(\xi \right)-f_{m}^{+} \left(\xi \right)\right\| _{L_{p,\rho } \left(\Gamma \right)} \to 0$, as  $n,m\to \infty $.

\noindent As the space $L_{p,\rho } \left(\Gamma \right)$ is complete, then

$\exists g\in L_{p,\rho } \left(\Gamma \right):f_{n}^{+} \left(\xi \right)\to g\left(\xi \right)$, $n\to \infty $, in  $L_{p,\rho } \left(\Gamma \right)$.

\noindent We have
\[\left\| f_{n} -f_{m} \right\| _{E_{1} \left(D^{+} \right)} =\left\| f_{n}^{+} -f_{m}^{+} \right\| _{L_{1} \left(\Gamma \right)} =\int _{\Gamma }\left|f_{n}^{+} \left(\xi \right)-f_{m}^{+} \left(\xi \right)\right| \left|d\xi \right|=\] 
\[=\int _{\Gamma }\left|f_{n}^{+} \left(\xi \right)-f_{m}^{+} \left(\xi \right)\right| \rho ^{\frac{1}{p} } \left(\xi \right)\rho ^{-\frac{1}{p} } \left(\xi \right)\left|d\xi \right|\le\]\[\le \left(\int _{\Gamma }\rho  ^{-\frac{q}{p} } \left(\xi \right)\left|d\xi \right|\right)^{\frac{1}{q} } \left(\int _{\Gamma }\left|f_{n}^{+} \left(\xi \right)-f_{m}^{+} \left(\xi \right)\right| ^{p} \rho \left(\xi \right)\left|d\xi \right|\right)^{\frac{1}{p} } .\] 
From $\rho ^{-\frac{q}{p} } \in L_{1} \left(\Gamma \right)$ it follows that $\left\{f_{n} \right\}_{n\in \mathbb N} $  is fundamental in $E_{1} \left(D^{+} \right)$ and hence, $\exists f\in E_{1} \left(D^{+} \right):f_{n} \to f,\, n\to \infty $, in $E_{1} \left(D^{+} \right)$. Hence, $f_{n}^{+} \left(\xi \right)\to f^{+} \left(\xi \right)\, ,\, n\to \infty $, in $L_{1} \left(\Gamma \right)$. Since
\[\left\| f_{n}^{+} -g\right\| _{L_{1} \left(\Gamma \right)} =\int _{\Gamma }\left|f_{n}^{+} \left(\xi \right)-g\left(\xi \right)\right| \left|d\xi \right|\le \] 
\[\le \left(\int _{\Gamma }\rho  ^{-\frac{q}{p} } \left(\xi \right)\left|d\xi \right|\right)^{\frac{1}{q} } \left(\int _{\Gamma }\left|f_{n}^{+} \left(\xi \right)-g\left(\xi \right)\right| ^{p} \rho \left(\xi \right)\left|d\xi \right|\right)^{\frac{1}{p} } ,\] 
from $f_{n}^{+} \to g\, ,\, n\to \infty $, in $L_{p,\rho } \left(\Gamma \right)$ it follows that $f_{n}^{+} \to g^{+} \, ,\, n\to \infty $, in $L_{1} \left(\Gamma \right)$. Hence we get that $f^{+} \left(\xi \right)=g\left(\xi \right)$ a.e. on $\Gamma $, hereby $\left\| f_{n} -f\right\| _{E_{p,\rho } \left(D^{+} \right)} \, ,\, n\to \infty $. That completes the proof.
\end{proof}

 The same reasoning proves the following  

\begin{lem}\label{L3.2} If $\rho ^{-\frac{q}{p} } \in L_{1} \left(\Gamma \right)$ then ${}_{m} E_{p,\rho } \left(D^{-} \right), 1<p<+\infty,$ is a Banach space.
\end{lem}

\section{The Riemann problem in weighted Smirnov classes}

Consider the following homogeneous Riemann problem in $E_{p,\rho } \left(D^{+} \right)\times _{m} E_{p,\rho } \left(D^{-} \right):$ 

\begin{equation} \label{GrindEQ__4_1_} 
A\left(\xi \right)F^{+} \left(\xi \right)+B\left(\xi \right)F^{-} \left(\xi \right)=0, \,\,\mbox{a.e.}\,\, \xi \in \Gamma.   
 \end{equation}
The solution of the problem \eqref{GrindEQ__4_1_}  is a pair of analytic functions 
\[\left(F^{+} ;F^{-} \right)\in E_{p,\rho } \left(D^{+} \right)\times _{m} E_{p,\rho } \left(D^{-} \right),\] 
whose non-tangential  boundary values $F^{\pm } \left(\xi \right)$  satisfy \eqref{GrindEQ__4_1_} on $\Gamma $. The ``non-weighted'' case of this problem was thoroughly studied earlier and we refer the reader to the monograph \cite{39} for its theory.

Let $S$ be the length of the curve $\Gamma $ and $z=z\left(s\right)\, ,\, 0\le s\le S,$ be the parametric equation of the curve $\Gamma $ in the element of length $ds$. Set
\[\sigma \left(s\right)\equiv \left|z\left(0\right)-z\left(s\right)\right|^{-\frac{h_{0} }{2\pi } } \prod _{0<s_{k} <S}\left|z\left(s_{k} \right)-z\left(s\right)\right| ^{-\frac{h_{k} }{2\pi } }  ,\] 
where $\xi _{k} =z\left(s_{k} \right)\, ,\, \, k=\overline{1,r}$. Let
\[\Omega \left(s\right)\equiv \theta \left(z\left(s\right)\right)\, ,\, 0\le s\le S.\] 
Hence 
\[h_{k} =\Omega \left(s_{k} +0\right)-\Omega \left(s_{k} -0\right), k=\overline{1,r}\, ;h_{0} =\Omega \left(+0\right)-\Omega \left(S-0\right). \] 
We will assume that the weight $\rho \left(\cdot \right)$  satisfies the following main condition:

$\alpha )$ \textit{  There exists $p_{1} ;p_{2} \in \left(1,+\infty \right)$, such that} 
\[\int _{0}^{S}\sigma ^{pp_{1} } \left(s\right)\rho ^{p_{1} } \left(z\left(s\right)\right)ds<+\infty  , \int _{0}^{S}\sigma ^{-qp_{2} } \left(s\right)\rho ^{-\frac{q}{p} p_{2} } \left(z\left(s\right)\right)ds<+\infty  .\]

In \cite{31} it was proved the following 

\begin{thm}\label{T4.1} Let the conditions i)-iii) and $\alpha )$ hold. Then, the general solution of the homogeneous problem  \eqref{GrindEQ__4_1_} in the class $E_{p,\rho } \left(D^{+} \right)\times _{m} E_{p,\rho } \left(D^{-} \right), 1<p<+\infty,$ has the form
\begin{equation} \label{GrindEQ__4_2_} 
F\left(z\right)\equiv Z\left(z\right)P_{m} \left(z\right),       
\end{equation} 
where $Z\left(z\right)$ is the canonical solution of the homogeneous problem, $P_{m} \left(z\right)$ is any polynomial of degree $k\le m.$ In the case $m<0$, $P_{m} \left(\cdot \right)\equiv 0$. 
\end{thm}

The canonical solution $Z\left(\, \cdot \, \right)$ is defined by $Z\left(z\right)=Z_{\left(\, \cdot \, \right)} \left(z\right)\tilde{Z}\left(z\right)$,  where

\[Z_{\left(1\right)} \left(z\right)=\exp \left\{\frac{1}{2\pi i} \int _{\Gamma }\ln \left|D\left(s\right)\right|\frac{dz\left(s\right)}{z\left(s\right)-z}  \right\} ,\] 
\[\tilde{Z}\left(z\right)=\exp \left\{\frac{1}{2\pi } \int _{\Gamma }\Omega \left(s\right)\frac{dz\left(s\right)}{z\left(s\right)-z}  \right\}.\] 
The function $D\left(\, \cdot \, \right)$  is defined by 
\[D\left(S\right)=-\frac{B\left(z\left(s\right)\right)}{A\left(z\left(s\right)\right)} , 0\le s\le S.\] 

 This theorem immediately implies

\begin{cor}\label{C4.1} Under the conditions of Theorem \ref{T4.1} the  problem \eqref{GrindEQ__4_1_} has only trivial solution in $E_{p,\rho } \left(D^{+} \right)\times _{m} E_{p,\rho } \left(D^{-} \right), 1<p<+\infty,$ if $F\left(\infty \right)=0$.
\end{cor}

Consider some special cases of the weight function  $\rho $.

\begin{ex}\label{E4.1} Let $\rho \left(\cdot \right)$ be 
\begin{equation} \label{GrindEQ__4_3_} 
\rho \left(z\left(s\right)\right)=\prod _{k=1}^{m}\left|s-t_{k} \right|^{\alpha _{k} }   ,             
\end{equation} 
where $\left\{t_{k} \right\}_{1}^{m} \subset \left(0,S\right)-$ are distinct points,  $\left\{\alpha _{k} \right\}_{1}^{m} \subset \mathbb R-$ are some numbers. Denote the union of $\left\{s_{k} \right\}_{0}^{r} $ and $\left\{t_{k} \right\}_{1}^{m} $ as $\left\{\tau _{k} \right\}_{1}^{l} :\left\{\tau _{k} \right\}_{1}^{l} \equiv \left\{s_{k} \right\}_{0}^{r} \cup \left\{t_{k} \right\}_{1}^{m} $. Let $\chi _{A} \left(\, \cdot \, \right)-$ be the indicator of the set $A$ and $T_{k} :T_{k} \equiv \left\{\tau _{k} \right\},$ is one point set  $k=\overline{1,l}$. Set 
\begin{equation} \label{GrindEQ__4_4_} 
\beta _{k} =-\frac{p}{2\pi } \sum _{i=1}^{r}h\_ _{i}  \chi _{T_{k} } \left(s_{i} \right)+\sum _{i=1}^{m}\alpha \_ _{i}  \chi _{T_{k} } \left(t_{i} \right),   \, k=\overline{1,l} .     
\end{equation} 
Assume that 
\begin{equation} \label{GrindEQ__4_5_} 
-1<\beta _{k} <\frac{p}{q} , k=\overline{1,l}.                 
\end{equation} 
It is easy to show that under \eqref{GrindEQ__4_5_} the condition $\alpha )$ holds. As a result we get the following
\end{ex}

\begin{cor}\label{C4.2}  Let the conditions i)-iii) hold and the weight $\rho \left(\cdot \right)$ has the form \eqref{GrindEQ__4_3_}.  Assume that \eqref{GrindEQ__4_5_} holds, where $\beta _{k} $ is defined as in \eqref{GrindEQ__4_4_}. Then the general solution of the homogeneous problem \eqref{GrindEQ__4_1_} in  $E_{p,\rho } \left(D^{+} \right)\times _{m} E_{p,\rho } \left(D^{-} \right), 1<p<+\infty,$ has the representation \eqref{GrindEQ__4_2_}.
\end{cor}

\begin{ex}\label{E4.2} As the weight $\rho \left(\cdot \right)$  we again take \eqref{GrindEQ__4_3_}, but now we assume that $\left\{s_{k} \right\}_{1}^{r} \cap \left\{t_{k} \right\}_{1}^{m} =\emptyset $. 
\end{ex}

In that case we get the following

\begin{cor}\label{C4.3}  Let the conditions of the Corollary \ref{C4.2} hold and $\left\{s_{k} \right\}_{1}^{r} \cap \left\{t_{k} \right\}_{1}^{m} =\emptyset $. If 
\[-\frac{1}{q} <\frac{h_{k} }{2\pi } <\frac{1}{p} \, ,\, k=\overline{1,r}\, ; -1<\alpha _{i} <\frac{q}{p} \, ,\, i=\overline{1,m}\, \, ,\] 
then the general solution of the problem \eqref{GrindEQ__4_1_} in $E_{p,\rho } \left(D^{+} \right)\times _{m} E_{p,\rho } \left(D^{-} \right), 1<p<+\infty,$ has the representation \eqref{GrindEQ__4_2_}.
\end{cor}

Now consider the non-homogeneous Riemann problem 
\begin{equation} \label{GrindEQ__4_6_} 
F^{+} \left(z\left(s\right)\right)-D\left(s\right)F^{-} \left(z\left(s\right)\right)=g\left(z\left(s\right)\right)\, ,\, s\in \left(0,S\right),         
\end{equation} 
where $g\in L_{p,\rho } \left(\Gamma \right)-$is a given function. The solution of the problem \eqref{GrindEQ__4_6_} is a pair of functions 
\[\left(F^{+} \left(z\right);F^{-} \left(z\right)\right)\in E_{p,\rho } \left(D^{+} \right)\times _{m} E_{p,\rho } \left(D^{-} \right),\] 
whose boundary values $F^{\pm } $ on $\Gamma $ a.e. hold \eqref{GrindEQ__4_6_}.

 In \cite{46} it was proved the following

\begin{thm}\label{T4.2}
  Let the conditions  i)-iii) hold. Given the weight $\rho \left(\cdot \right)$  of the form \eqref{GrindEQ__4_3_} and the numbers  $\left\{\beta _{k} \right\}_{1}^{l} $ are defined   as in \eqref{GrindEQ__4_4_}. Let \eqref{GrindEQ__4_5_}
   holds and  $\alpha _{k} <\frac{q}{p} \, ,\, k=\overline{1,m}$. Then the general solution of the problem \eqref{GrindEQ__4_6_} in $E_{p,\rho } \left(D^{+} \right)\times _{m} E_{p,\rho } \left(D^{-} \right), 1<p<+\infty,$ has the form 
\[F\left(z\right)=F_{0} \left(z\right)+F_{1} \left(z\right),\] 
where $F_{0} \left(z\right)-$ is a general solution of the corresponding homogeneous problem and $F_{1} \left(z\right)$ is expressed as 
\begin{equation} \label{GrindEQ__4_7_} 
F_{1} \left(z\right)\equiv \frac{Z\left(z\right)}{2\pi i} \int _{\Gamma }\frac{g\left(\xi \right)d\xi }{Z^{+} \left(\xi \right)\left(\xi -z\right)},              
\end{equation} 
where $Z\left(z\right)-$ is the canonical solution, $m\ge 0-$is an integer.
\end{thm}

This theorem implies the following 

\begin{cor}\label{C4.4} Let all conditions of Theorem \ref{T4.2} hold. If $F\left(\infty \right)=0$ then the problem \eqref{GrindEQ__4_6_} has a unique solution in the class $E_{p,\rho } \left(D^{+} \right)\times _{m} E_{p,\rho } \left(D^{-} \right), 1<p<+\infty,$ expressed as in \eqref{GrindEQ__4_7_}.
\end{cor}

\section{Basisness of $p$-Faber polynomials in the weighted Smirnov spaces}

In this section the basisness of generalized Faber polynomials in weighted Smirnov spaces is established under the condition that the weight and the boundary values of the conformal mapping performing an isomorphism between the unit circle and the simply-connected region under consideration satisfy the Muckenhoupt condition.

 Thus for the function $f\in L_{p;\rho } \left(\Gamma \right)$ we set
\[f_{+} \left(w\right)=f\left[\varphi _{-1} \left(w\right)\right]\left(\varphi '_{-1} \left(w\right)\right)^{\frac{1}{p} } , f_{-} \left(w\right)=f\left[\psi _{-1} \left(w\right)\right]\left(\psi '_{-1} \left(w\right)\right)^{\frac{2}{p} } \, ,\, \, w\in \partial \omega ,\]
and for weight function $\rho\left(\cdot\right)$ on $\Gamma$ define the following weight functions on $\partial \omega$:

\begin{equation} \label{GrindEQ__5_0_} 
 \rho _{+}\left(w\right)=\rho \left[\varphi_{-1}\left(w\right)\right]; \rho _{-}\left(w\right)=\rho \left[\psi_{-1}\left(w\right)\right], w \in \partial \omega
 \end{equation}
 
Consider the following operator $T_{p}^{+} $:
\[\left(T_{p}^{+} f\right)\left(z\right)=\frac{1}{2\pi i} \int _{\partial \omega }\frac{f\left(\xi \right)\left[\varphi '_{-1} \left(\xi \right)\right]^{\frac{1}{q} } }{\varphi _{-1} \left(\xi \right)-z}  \, \, d\xi \, ,\, \, z\in D,\] 
where $f\in H_{p,\rho _{+} }^{+} $.  In sequel we will use the following theorem which is similarly proved  as the Morrey-Hardy case   considered in the  paper \cite{6}.

\begin{thm}\label{T5.1} Let $\Gamma -$be a regular curve and $1<p<+\infty $. If $\rho \in A_{p} \left(\Gamma \right)$ and $\rho _{+} \in A_{p} \left(T\right)$,  then $T_{p}^{+} :H_{p,\rho _{+} }^{+} \leftrightarrow E_{p,\rho } \left(D^{+} \right)$ is a one-to-one operator onto $E_{p,\rho } \left(D^{+} \right)$.
\end{thm}

Using this theorem we can  prove the following 

\begin{thm}\label{T5.2} Let $\Gamma $ be a regular curve and $0\in int$ $ \Gamma$. If $\rho _{\pm } \left(\cdot \right)\in A_{p} \left(T\right)$, $\, \rho \left(\cdot \right)\in A_{p} \left(\Gamma \right)$, $1<p<+\infty $, then the system of generalized $p$-Faber polynomials $\left\{F_{p,n}^{+} \right\}_{n\in \mathbb Z_{+} } $ and $\left\{F_{p,n}^{-} \right\}_{n\in \mathbb N} $ form a basis of the spaces $E_{p,\rho } \left(D^{+} \right)$ and ${}_{-1} E_{p,\rho } \left(D^{-} \right), 1<p<+\infty$, respectively. 
\end{thm}

\begin{proof} It is easy to see that
\[\left(T_{p}^{+} f_{n} \right)\left(z\right)=F_{p,n}^{+} \left(z\right)\, ,\, \forall n\in \mathbb Z_{+} ,\, \] 
where $f_{n} \left(z\right)=z^{n}.$ We next turn our attention to Theorem \ref{T2.1}. Under the conditions of this theorem  the weight function $\rho _{+} \left(\cdot \right)$ satisfies $A_{p} \left(T\right)$, $1<p<+\infty,$ Muckenhoupt condition on the unit circle $T$.  Then by Theorem \ref{T2.1}, the system $\left\{z^{n} \right\}_{n\in \mathbb Z_{+} } $ forms a basis in $H_{p,\rho _{+} }^{+} $. Theorem \ref{D}  implies that if $\rho \in A_{p} \left(\Gamma \right)$, then the system $\left\{F_{p,n}^{+} \right\}_{n\in \mathbb Z_{+} } $ forms a basis in $E_{p,\rho } \left(D\right)$.

Denote by $L_{p,\rho }^{+} \left(\Gamma\right)$ the set of restrictions of the functions of $E_{p,\rho } \left(D\right)$ to $\Gamma $. We get that the system $\left\{F_{p,n}^{+} \left(\xi \right)\right\}_{n\in \mathbb Z_{+} } $ forms a basis in $L_{p,\rho }^{+} \left(\Gamma\right)$, if $ $ is a regular curve and $\rho _{+} \in A_{p} \left(T\right)$ \& $\rho \in A_{p} \left(\Gamma \right)$.

Show that analogous reasoning works also for the system $\left\{F_{p,n}^{-} \right\}_{n\in \mathbb N} $. Indeed, denote by $L_{p}^{-} \left(\Gamma\right)$ the restriction of $E_{p} \left(D^{-} \right)$ to $\Gamma $. Prove that $\left\{F_{p,n}^{-} \right\}_{n\in \mathbb  N} $ forms a basis in $L_{p,\rho }^{-} \left(\Gamma\right)$ under the conditions of theorem on functions $\rho \left(\cdot \right)$ and $\rho _{-} \left(\cdot \right)$. Let us take any polynomial $P\left(z^{-1} \right)=a_{1} \, z^{-1} +...+a_{r} \, z^{-r} $ and define the operator $T_{p}^{-} $ as follows
\begin{equation} \label{GrindEQ__5_1_} 
T_{p}^{-} P\left(z^{-1} \right)=a_{1} F_{p,\, 1}^{-} \left(z^{-1} \right)+...+a_{r} F_{p,\, r}^{-} \left(z^{-1} \right).           
\end{equation} 
Since the function $P\left(z^{-1} \right)$  is bounded in $C\backslash \omega $, from the definition of ${}_{-1} H_{p,\rho _{-} }^{-} $ it immediately follows that $P\left(z^{-1} \right)\in _{-1} H_{p,\rho _{-} }^{-} $. Since $\rho _{-} \left(\cdot \right)\in A_{p} \left(T\right)$, from the Theorem \ref{T2.1} ii) it follows that  the system  $\left\{z^{-n} \right\}_{n\in \mathbb N} $  forms a basis in ${}_{-1} H_{p,\rho _{-} }^{-} $, which proves that the set of polynomials $P\left(z^{-1} \right)$  are dense in ${}_{-1} H_{p,\rho _{-} }^{-} $. Denote the set of polynomials of the form $P\left(z^{-1} \right)$ by ${\Pi }^{-} $. Thus, ${\Pi }^{-} \subset _{-1} H_{p,\rho _{-} }^{-} $ and  $\overline{{\Pi }^{-} }=_{-1} H_{p,\rho _{-} }^{-} $ ($\bar{M}-$ is the closure of the set $ M$ in ${}_{-1} H_{p,\rho _{-} }^{-} $). Show that $T_{p}^{-} P\left(z^{-1} \right)\in _{-1} E_{p,\rho }^{} \left(D^{-} \right)$. Due to \eqref{GrindEQ__5_1_} it is enough to show that $F_{p,n}^{-} \in _{-1} E_{p,\rho }^{} \left(D^{-} \right),\, \, \forall n\in \mathbb N$. It is evident that there exists a number $M_{n} >0$ for which
\[\left|F_{p,n}^{-} \left(z\right)\right|\le M_{n} \, \, ,\, \, \, \forall z\in D^{-} .\] 
Since,  $F_{p,n}^{-} \left(z\right)$  is a polynomial from $z^{-1} $ of order $n$.  We have
\[\left\| F_{p,n}^{-} \right\| _{{}_{-1} E_{p,\rho } \left(D^{-} \right)} =\left\| F_{p,n}^{-} \left(\xi \right)\right\| _{L_{p,\rho } \left(\Gamma \right)} =\left(\int _{\Gamma }\left|F_{p,n}^{-} \left(\xi \right)\right|^{p} \rho \left(\xi \right)d\xi  \right)^{\frac{1}{p} } \le \] 
\[\le M_{n} \left(\int _{\Gamma }\rho \left(\xi \right)d\xi  \right)^{\frac{1}{p} } <+\infty ,\] 
since, $\rho \left(\cdot \right)\in A_{p} \left(\Gamma \right)$. The inclusion $F_{p,n}^{-} \in _{-1} E_{p} \left(D^{-} \right)$ is clear. Then, from the above expression it follows that $F_{p,n}^{-} \in _{-1} E_{p,\rho }^{} \left(D^{-} \right)$.  Hence, $T_{p}^{-} :\Pi ^{-} \to _{-1} E_{p,\rho }^{} \left(D^{-} \right)$.  Similarly  to the case $T_{p}^{+} $, it is proved that the operator $T_{p}^{-} $ is bounded in ${\Pi }^{-} $.  Indeed, consider the integral expression for $F_{p,n}^{-} $:
\begin{equation} \label{GrindEQ__5_2_} 
F_{p,\, n}^{-} \left(z^{-1} \right)=\frac{1}{2\pi i} \int _{\left|w\right|=R}\frac{w^{n-\frac{p}{2} } \left[\psi '_{-1} \left(w\right)\right]^{1-\frac{1}{p} } }{\psi _{-1} \left(w\right)-z} \,dw\, .\,  
\end{equation} 
Now,  let us  take an arbitrary polynomial of $r$th order $P_{r} $:
\[P_{r} \left(z^{-1} \right)=a_{1} z^{-1} +...+a_{r} z^{-r} ,\] 
and oppose to it the function 
\[\left[T_{p} P_{r} \right]\left(z\right)=\frac{1}{2\pi i} \int _{\left|w\right|=R}\frac{P_{r} \left(w\right)\left[\psi '_{-1} \left(w\right)\right]^{1-\frac{1}{p} } }{w^{\frac{p}{2} } \left(\psi _{-1} \left(w\right)-z\right)}  \, \, dw\, .\] 
We have 
\[\left[T_{p} P_{r} \right]\left(z\right)=\sum _{k=1}^{r}a_{k}  \int _{\left|w\right|=R}\frac{w^{k-\frac{p}{2} } \left[\psi '_{-1} \left(w\right)\right]^{1-\frac{1}{p} } }{\psi _{-1} \left(w\right)-z}  \, \, dw\, =\sum _{k=1}^{r}a_{k} F_{p,k}^{-} \left(z^{-1} \right) ,\] 
i.e.  the representation \eqref{GrindEQ__5_2_} coincides with the expression on $\Pi ^{-} $ . It is clear that $T_{p}^{-} $ is a linear operator acting from $\Pi ^{-} $ to ${}_{-1} E_{p} \left(D^{-} \right)$, i.e. $T_{p}^{-} :\, \Pi ^{-} \to _{-1} E_{p} \left(D^{-} \right)$. Let us show that  $T_{p}^{-} $  is bounded on ${\Pi }^{-} $. It is not difficult to see that  an integral representation  is also true  for $T_{p}^{-} $  and in the case $R=1$, i.e.
\begin{equation} \label{GrindEQ__5_3_} 
\left[T_{p}^{-} P_{r} \right]\left(z\right)=\frac{1}{2\pi i} \int _{\left|w\right|=1}\frac{P_{r} \left(w\right)\left[\psi '_{-1} \left(w\right)\right]^{1-\frac{1}{p} } }{w^{\frac{2}{p} } \left(\psi _{-1} \left(w\right)-z\right)}  \, \, dw\, .       
\end{equation} 
This follows from the Cauchy integral formula for functions from the  Smirnov class.

Making a substitution $\xi =\psi _{-1} \left(w\right)$  in the integral expression, we get
\[\left[T_{p}^{-} P_{r} \right]\left(z\right)=\frac{1}{2\pi i} \int _{}\frac{P_{r} \left[\psi \left(\xi \right)\right]\, \psi ^{-\frac{2}{p} } \left(\xi \right)\left[\psi '\left(\xi \right)\right]^{\frac{1}{p} } }{\xi -z} \,  d\xi \, ,\] 
where $z\in ext\, \Gamma =D^{-} $. Letting $z$  to $\tau \in \Gamma$ and applying the formula Sokhotskii-Plemelj, we have  
\begin{equation} \label{GrindEQ__5_4_} 
\left[T_{p}^{-} P_{r} \right]\left(\tau \right)=-\frac{1}{2} \tilde{P}_{r} \left(\tau \right)+\left[S_{\Gamma } \tilde{P}_{r} \right]\left(\tau \right),       
\end{equation} 
where
\[\tilde{P}_{r} \left(\tau \right)=P_{r} \left[\psi \left(\tau \right)\right]\, \psi ^{-\frac{2}{p} } \left(\tau \right)\left[\psi '\left(\tau \right)\right]^{\frac{1}{p} } .\] 
Taking the norm $\left\| \, \cdot \, \right\| _{L_{p,\rho } \left(\Gamma \right)} $, hence we obtain
\[\left\| T_{p}^{-} P_{r} \right\| _{L_{p,\rho } \left(\Gamma\right)} \le \frac{1}{2} \left\| \tilde{P}_{r} \left(\tau \right)\right\| _{L_{p,\rho } \left(\Gamma \right)} +\left\| S\tilde{P}_{r} \right\| _{L_{p,\rho } \left(\Gamma\right)} .\] 
We have 
\[\left\| \tilde{P}_{r} \right\| _{L_{p,\rho } \left(\Gamma\right)}^{p} =\int _{\Gamma}\left|\tilde{P}_{r} \left(\tau \right)\right|^{p} \rho \left(\tau \right)\left|d\tau \right|= \int _{\Gamma}\left|P_{r} \left[\left(\psi \left(\tau \right)\right)^{-1} \right]\right|^{p} \, \left|\psi ^{-2} \left(\tau \right)\right|\rho \left(\tau \right)\, \left|d\psi \left(\tau \right)\right|= \] 
\[=\int _{\left|w\right|=1}\left|P_{r} \left(w^{-1} \right)\right|^{p} \left|w^{-2} \right|\rho \left(\psi _{-1} \left(w\right)\right)\left|dw\right|=\int _{\left|w\right|=1}\left|P_{r} \left(w^{-1} \right)\right|^{p} \rho _{-} \left(w\right)\left|dw\right|=  \] 
\[=\left\| P_{r} \left(w^{-1} \right)\right\| _{L_{p,\rho _{-} } \left(\partial \omega \right)}^{} <+\infty .\] 
Since $\rho \left(\cdot \right)\in A_{p} \left(\Gamma \right)$, then from the boundedness of the singular operator $S$ in $L_{p,\rho } \left(\Gamma\right)$, from the relation \eqref{GrindEQ__5_4_}, we obtain
\[\left\| T_{p}^{-} P_{r} \right\| _{L_{p,\rho } \left(\Gamma \right)} \le c\left\| \tilde{P}_{r} \right\| _{L_{p,\rho } \left(\Gamma \right)} \le c\left\| P_{r} \left(w^{-1} \right)\right\| _{L_{p,\rho _{-} } \left(\partial \omega \right)} .\] 
Consequently, the operator $T_{p}^{-} $ is bounded on ${\Pi }^{-} $.  Extend $T_{p}^{-} $ by continuity $T_{p}^{-} $  to ${}_{-1} H_{p,\rho _{-} }^{-} $ and denote it again by  $T_{p}^{-} $. Thus, we obtain that $T_{p}^{-} $  boundedly acts from ${}_{-1} H_{p,\rho _{-} }^{-} $ to  ${}_{-1} E_{p,\rho } \left(D^{-} \right)$. Let us show that  $T_{p}^{-} $ establishes an isomorphism between ${}_{-1} H_{p,\rho _{-} }^{-} $   and ${}_{-1} E_{p,\rho } \left(D^{-} \right)$. It follows directly from \eqref{GrindEQ__5_3_}:
\[\left[T_{p}^{-} f\right]\left(z\right)=\frac{1}{2\pi i} \int _{\left|w\right|=1}\frac{f^{-} \left(w^{-1} \right)\left[\psi '_{-1} \left(w\right)\right]^{1-\frac{1}{p} } }{w^{\frac{2}{p} } \left(\psi _{-1} \left(w\right)-z\right)}  \, \, dw\, =\] 
\begin{equation} \label{GrindEQ__5_5_} 
=\frac{1}{2\pi i} \int _{\Gamma}\frac{f^{-} \left[\left(\psi \left(\xi \right)\right)^{-1} \right]\, \psi ^{-\frac{2}{p} } \left(\xi \right)\left[\psi '\left(\xi \right)\right]^{\frac{1}{p} } }{\xi -z} \, \,  d\xi \, .         
\end{equation} 
Hence we obtain that if $\left[T_{p}^{-} f\right]\left(z\right)\equiv 0$,  $\forall z\in ext \Gamma\, $, then   $f\equiv 0$, i.e. $KerT_{p}^{-} =\left\{0\right\}$. 

Take $\forall F\in $ ${}_{-1} E_{p,\rho } \left(D^{-} \right)$. Since $F\left(\infty \right)=0$, then from the Cauchy formula we obtain 
\[F\left(z\right)=-\frac{1}{2\pi i} \int _{\Gamma}\frac{F^{-} \left(\xi \right)d\xi }{\xi -z} \,  \, \, ,\, \, z\in D^{-} ,\] 
where $F^{-} \left(\cdot \right)-$is  nontangential boundary values of the function $F\left(\cdot \right)$ on $\Gamma $. Assume
\begin{equation} \label{GrindEQ__5_6_} 
g\left(w\right)=F\left[\psi _{-1} \left(w\right)\right]w^{\frac{2}{p} } \left[\psi '_{-1} \left(w\right)\right]^{\frac{1}{p} } \, ,\, \left|w\right|<1.        
\end{equation} 

$\psi _{-1} \left(w\right)$ in a neighborhood of a point $w=0$ has a representation of the form 
\[\psi _{-1} \left(w\right)=\beta \, w^{-1} +\beta _{0} +\beta _{1} w+...\, ,\] 
where $\beta $ $\ne 0$. Consequently,  $w^{2} \psi '_{-1} \left(w\right)$ is an analytic function for  $\left|w\right|<1$ and  ${\mathop{\lim }\limits_{w\to 0}} $ $w^{2} \psi '_{-1} \left(w\right)$ $=\beta $ $\ne 0$. Thus, the function $g\left(w\right)$ is analytic for  $\left|w\right|<1$. 

 Let us show that $g\in H_{p,\rho _{-} }^{+} $. It follows immediately from \eqref{GrindEQ__5_6_} that the nontangential boundary values $g^{+} \left(\xi \right)$ on $\partial \omega $  are connected by non-tangent boundary values $F^{-} \left[\psi _{-1} \left(\xi \right)\right]$  by the relation 
\[g^{+} \left(\xi \right)=F^{-} \left[\psi _{-1} \left(\xi \right)\right]\left[\psi '_{-1} \left(\xi \right)\right]^{\frac{1}{p} } \xi ^{\frac{2}{p} } \, ,\, \, \left|\xi \right|=1.\] 
We have
\[\int _{\left|\xi \right|=1}\left|g^{+} \left(\xi \right)\right|^{p} \rho _{-} \left(\xi \right)\left|d\xi \right| =\int _{\left|\xi \right|=1}\left|F^{-} \left[\psi _{-1} \left(\xi \right)\right]\, \right|^{p} \left|\psi '_{-1} \left(\xi \right)\, \right|\rho _{-} \left(\xi \right)\, \left|d\xi \right|= \, \] 
\[=/\xi =\psi \left(\tau \right)\, ,\, \tau \in \Gamma \, /=\int _{\Gamma}\left|F^{-} \left(\tau \right)\right|^{p} \left|\psi '_{-1} \left(\psi \left(\tau \right)\right)\right|\rho _{-} \left(\psi \left(\tau \right)\right)\left|\psi '\left(\tau \right)\right|\left|d\tau \right| .\] 
Given here the obvious expressions
\[\frac{d\psi }{d\tau } =\frac{1}{{\frac{d\psi _{-1} \left(\xi \right)}{d\xi }  \mathord{\left/{\vphantom{\frac{d\psi _{-1} \left(\xi \right)}{d\xi }  \begin{array}{l} {} \\ {{}_{\xi =\psi \left(\tau \right)} } \end{array}}}\right.\kern-\nulldelimiterspace} \begin{array}{l} {} \\ {{}_{\xi =\psi \left(\tau \right)} } \end{array}} } ,\] 
and 
\[\rho _{-} \left(\xi \right)=\rho \left[\psi _{-1} \left(\xi \right)\right]\, ,\xi \in T\Leftrightarrow \rho _{-} \left(\psi \left(\tau \right)\right)=\rho \left(\tau \right),\, \, \tau \in \Gamma ,\] 
we have  
\[\int _{\left|\xi \right|=1}\left|g^{+} \left(\xi \right)\right|^{p} \, \rho _{-} \left(\xi \right)\left|d\xi \right| =\int _{\Gamma}\left|F^{-} \left(\tau \right)\right|^{p} \rho \left(\tau \right)\left|d\tau \right| <+\infty ,\] 
since,  $F\in $ ${}_{-1} E_{p,\rho } \left(D^{-} \right)$.

 Consequently, $g^{+} \left(\cdot \right)\in $ $L_{p,\rho _{-} } \left(\partial \omega \right)$. Let us show that $g\left(\cdot \right)\in H_{1}^{+} $. First, we show that for sufficiently small  $\delta >0$ the inclusion $g\left(\cdot \right)\in H_{\delta }^{+} $ is true. Let $p_{1} >2$  be some number and  $q_{1} :\frac{1}{p_{1} } +\frac{1}{q_{1} } =1-$ is a conjugate number  to a number  $p_{1} $. Let   $\omega _{r} =\left\{z\in \mathbb C:\left|z\right|=r\right\}\, ,\, \, \, 0<r<1$,  and $\Gamma _{r} =\psi _{-1} \left(\omega _{r} \right)-\psi _{-1} \left(\cdot \right)$ is an image of $\omega_{r}$. We have
\[\int _{\omega _{r} }\left|g\left(w\right)\right|^{\delta }  \left|dw\right|=r^{2\frac{\delta }{p} } \int _{\omega _{r} }\left|F\left(\psi _{-1} \left(w\right)\right)\right|^{\delta } \left|\psi '_{-1} \left(w\right)\right|^{\frac{\delta }{p} }  \left|dw\right|=\] 
\[=r^{2\frac{\delta }{p} } \int _{\omega _{r} }\left|F\left(\psi _{-1} \left(w\right)\right)\right|^{\delta } \left|\psi '_{-1} \left(w\right)\right|^{-\frac{1}{p_{1} } }  \left|\psi '_{-1} \left(w\right)\right|^{\frac{\delta }{p} +\frac{1}{p_{1} } } \left|dw\right|\le \] 
\[\le \left(\int _{\omega _{r} }\left|F\left(\psi _{-1} \left(w\right)\right)\right|^{\delta p_{1} } \left|\psi '_{-1} \left(w\right)\right|^{-1}  \left|dw\right|\right)^{\frac{1}{p_{1} } } \left(\int _{\omega _{r} }\left|\psi '_{-1} \left(w\right)\right|^{\left(\frac{\delta }{p} +\frac{1}{p_{1} } \right)q_{1} }  \left|dw\right|\right)^{\frac{1}{q_{1} } } =\]\[=I_{1} \left(r\right)I_{2} \left(r\right).\] 
Let  $\delta >0$ be satisfy the condition 
\begin{equation} \label{GrindEQ__5_7_} 
0<\delta <\min \left\{\frac{1}{p_{1} } ;p\left(1-\frac{2}{p_{1} } \right)\right\}.       
\end{equation} 
Thus
\[I_{1}^{p_{1} } \left(r\right)=\int _{\omega _{r} }\left|F\left(\psi _{-1} \left(w\right)\right)\right|^{\delta p_{1} } \left|\psi '_{-1} \left(w\right)\right|^{-1} \left|dw\right| =\] 
\[=/w=\psi \left(z\right)/=\int _{\Gamma _{r} }\left|F\left(z\right)\right|^{\delta p_{1} } \left|\psi '_{-1} \left(\psi \left(z\right)\right)\right|^{-1} \left|\psi '\left(z\right)\right|\left|dz\right| .\] 
Taking into account the obvious relation
\[\psi '_{-1} \left(\psi \left(z\right)\right)=\frac{1}{\psi '\left(z\right)} ,\] 
we have
\begin{equation} \label{GrindEQ__5_8_} 
I_{1}^{p_{1} } \left(r\right)=\int _{\Gamma _{r} }\left|F\left(z\right)\right|^{\delta p_{1} }  \left|dz\right|.          
\end{equation} 
Since, $F\in _{-1} E_{1} \left(D^{-} \right)$, it is clear that  $F\in _{-1} E_{\delta _{1} } \left(D^{-} \right)$,  for  $\forall \delta _{1} >0:\delta _{1} <1$. Then from \eqref{GrindEQ__5_8_} we obtain
\[{\mathop{\sup }\limits_{0<r<1}} I_{1}^{p_{1} } \left(r\right)<+\infty ,\] 
so, $\delta _{1} =\delta p_{1} <1$.

 As is known,  $\psi '_{-1} \left(\cdot \right)\in H_{1}^{+} $ (see e.g.  Goluzin  \cite[ pp. 405]{45}). Consequently,  $\psi '_{-1} \left(\cdot \right)\in H_{\delta _{2} }^{+} $, $\forall \delta _{2} \in \left(0,1\right)$. It follows immediately from the condition \eqref{GrindEQ__5_7_} that  $\delta _{2} =q_{1} \left(\frac{\delta }{p} +\frac{1}{p_{1} } \right)<1$. Then we have 
\[{\mathop{\sup }\limits_{0<r<1}} I_{2}^{q_{1} } \left(r\right)={\mathop{\sup }\limits_{0<r<1}} \int _{\omega _{r} }\left|\psi '_{-1} \left(w\right)\right|^{\delta _{2} }  \left|dw\right|<+\infty .\] 
As a result, we find that if $\delta $  satisfies the condition \eqref{GrindEQ__5_7_}, then
\[{\mathop{\sup }\limits_{0<r<1}} \int _{\omega _{r} }\left|g\left(w\right)\right|^{\delta }  \left|dw\right|<+\infty ,\] 
i.e.   $g\in H_{\delta }^{+} $.

Now show that the boundary values  $g^{+} \left(\cdot \right)$ of the function $g\left(\cdot \right)$ belong to the class $L_{1} \left(\partial \omega \right)$, i.e.  $g^{+} \left(\cdot \right)\in $ $L_{1} \left(\partial \omega \right)$. From \eqref{GrindEQ__5_6_} we immediately obtain that the nontangential boundary values $g^{+} \left(\cdot \right)$ of a function $g\left(\cdot \right)$ on the inside $\omega $  are expressed by nontangential boundary values $F^{-} \left(\cdot \right)$ on  $\partial \omega $ the outside $\omega $, by the relation
\[g^{+} \left(\xi \right)=F^{-} \left(\psi _{-1} \left(\xi \right)\right)\left(\psi '_{-1} \left(\xi \right)\right)^{\frac{1}{p} } \xi ^{\frac{2}{p} } ,\, \, \, \xi \in \partial \omega .\] 
We have
\[\int _{\left|\xi \right|=1}\left|g^{+} \left(\xi \right)\right| \left|d\xi \right|=\int _{\left|\xi \right|=1}\left|F^{-} \left(\psi _{-1} \left(\xi \right)\right)\right|\, \left|\psi '_{-1} \left(\xi \right)\right|^{\frac{1}{p} } \,  \left|d\xi \right|=\] 
\[=\int _{\left|\xi \right|=1}\left|F^{-} \left(\psi _{-1} \left(\xi \right)\right)\right|\, \rho _{-}^{\frac{1}{p} } \left(\xi \right)\left|\psi '_{-1} \left(\xi \right)\right|^{\frac{1}{p} } \, \rho _{-}^{-\frac{1}{p} } \left(\xi \right) \left|d\xi \right|\le \] 
\[\le \left(\int _{\left|\xi \right|=1}\left|F^{-} \left(\psi _{-1} \left(\xi \right)\right)\right|^{p} \, \rho _{-}^{} \left(\xi \right)\left|\psi '_{-1} \left(\xi \right)\right|\,  \left|d\xi \right|\right)^{\frac{1}{p} } \, \, \left(\int _{\left|\xi \right|=1}\, \rho _{-}^{-\frac{q}{p} } \left(\xi \right)\,  \left|d\xi \right|\right)^{\frac{1}{q} } =I_{1} I_{2},\] 
where $I_{1}$  and $I_{2}$  denotes the corresponding integrals. Paying attention to the expression
\[{\frac{d\psi _{-1} \left(\xi \right)}{d\xi }  \mathord{\left/{\vphantom{\frac{d\psi _{-1} \left(\xi \right)}{d\xi }  \begin{array}{l} {} \\ {{}_{\xi =\psi \left(\tau \right)} } \end{array}}}\right.\kern-\nulldelimiterspace} \begin{array}{l} {} \\ {{}_{\xi =\psi \left(\tau \right)} } \end{array}} =\left(\frac{d\psi }{d\tau } \right)^{-1} ,\] 
we obtain 
\[I_{1}^{p} =\int _{\left|\xi \right|=1}\left|F^{-} \left(\psi _{-1} \left(\xi \right)\right)\right|^{p} \rho _{-} \left(\xi \right)\, \left|\psi '_{-1} \left(\xi \right)\right|\,  \left|d\xi \right|=/\xi =\psi \left(\tau \right)/=\] 
\[=\int _{\Gamma }\left|F^{-} \left(\tau \right)\right|^{p} \rho _{-} \left(\psi \left(\tau \right)\right)\, \,  \left|d\tau \right|=\int _{\Gamma }\left|F^{-} \left(\tau \right)\right|^{p} \rho \left(\tau \right)\, \,  \left|d\tau \right|<+\infty ,\] 
since, $F^{-} \left(\cdot \right)\in L_{p,\rho } \left(\Gamma \right)$.
\[I_{2}^{q} =\int _{\left|\xi \right|=1}\rho ^{-\frac{q}{p} } \, \left(\psi _{-1} \left(\xi \right)\right)\,  \left|d\xi \right|<+\infty ,\] 
since, by assumption  we have $\rho _{-} \left(\cdot \right)\in A_{p} \left(\partial \omega \right)$. Combining these results, we obtain that $g^{+} \left(\cdot \right)\in L_{1} \left(\partial \omega \right)$.  Then from the Smirnov  theorem   it follows  that  $g\left(\cdot \right)\in H_{1}^{+} $. Since, $g^{+} \left(\cdot \right)\in L_{p,\rho _{-} } \left(\partial \omega \right)$, then by definition we obtain that $g\left(\cdot \right)\in H_{p,\rho _{-} }^{+} $. On the other hand, from the representation \eqref{GrindEQ__5_6_} we have
\[g\left(0\right)=F\left(\infty \right)\left({\mathop{\lim }\limits_{w\to 0}} w^{2} \psi '_{-1} \left(w\right)\right)^{\frac{1}{p} } =0\, \cdot \, \beta ^{\frac{1}{p} } =0,    \] 
as $F\left(\infty \right)=0$.  Thus, the Taylor expansion of a function $g\left(\cdot \right)$ at zero has the form
\[g\left(w\right)=a_{1} w+a_{2} w^{2} +...  .\] 
Accept
\[f\left(z\right)=-g\left(\frac{1}{z} \right)\, ,\, \, \left|z\right|>1.\] 
It is not difficult to see that nontangential boundary values  $f^{-} \left(\cdot \right)$  of the function $f\left(\cdot \right)$  from outside  $\omega $ to   $\partial \omega $  are  equal $f^{-} \left(\xi \right)=-g^{+} \left(\bar{\xi }\right)\, ,\, \, \left|\xi \right|=1$. Consequently,  $f\left(\infty \right)=-g\left(0\right)=0$. Using these relations, it is not difficult to establish the inclusion  $f\in _{-1} H_{p,\rho _{-} }^{-} $. Paying   attention to the expression
\[f^{-} \left(w^{-1} \right)=-g^{+} \left(w\right)\, ,\, \, w\in \partial \omega ,\] 
from the representation \eqref{GrindEQ__5_5_} for the operator  $T_{p}^{-} $ we obtain
\[\left[T_{p}^{-} f\right]\left(z\right)=\frac{1}{2\pi i} \int _{\left|w\right|=1}\frac{f^{-} \left(w^{-1} \right)\left(\psi '_{-1} \left(z\right)\right)^{1-\frac{1}{p} } }{w^{\frac{2}{p} } \left(\psi _{-1} \left(w\right)-z\right)}  dw=\] 
\[=-\frac{1}{2\pi i} \int _{\left|w\right|=1}\frac{F^{-} \left(\psi _{-1} \left(w\right)\right)w^{\frac{2}{p} } \left(\psi '_{-1} \left(z\right)\right)^{\frac{1}{p} } \left(\psi '_{-1} \left(z\right)\right)^{1-\frac{1}{p} } }{w^{\frac{2}{p} } \left(\psi _{-1} \left(w\right)-z\right)}  dw=\] 

\[=/w=\psi \left(\xi \right)/=-\frac{1}{2\pi i} \int _{\Gamma }\frac{F^{-} \left(\xi \right)\psi '_{-1} \left(\psi \left(\xi \right)\right)\psi '\left(\xi \right)}{\xi -z}  d\xi =\] 
\[=-\frac{1}{2\pi i} \int _{\Gamma }\frac{F^{-} \left(\xi \right)}{\xi -z}  d\xi =F\left(z\right)\, ,\, \, \forall z\in D^{-} .\] 

 So,  $\left[T_{p}^{-} f\right]\left(z\right)=F\left(z\right)$. Thus,  we obtain that  the operator  $T_{p}^{-} $  boundedly  maps ${}_{-1} H_{p,\rho _{-} }^{-} $ onto ${}_{-1} E_{p,\rho } \left(D^{-} \right)$.  Therefore, by the Banach's theorem it follows that $\left(T_{p}^{-} \right)^{-1} $ is also bounded and as a result,   $T_{p}^{-} $ is an isomorphism between ${}_{-1} H_{p,\rho _{-} }^{-} $ and ${}_{-1} E_{p,\rho } \left(D^{-} \right)$. Furthermore, $\left[T_{p}^{-} \left(\xi ^{-n} \right)\right]\left(z\right)=F_{p,n}^{-} \left(z\right)\, ,\, \, n\in \mathbb N$. Since $\rho _{-} \left(\cdot \right)\in A_{p} \left(T\right)$, by Theorem \ref{T2.1}, the system $\left\{z^{-n} \right\}_{n\in \mathbb N} $ forms a basis in ${}_{-1} H_{p,\rho _{-} }^{-} $. As e result, the system $\left\{F_{p,n}^{-} \left(z\right)\right\}_{n\in \mathbb N} $  forms a basis in ${}_{-1} E_{p,\rho } \left(D^{-} \right)$. Thus,  $\forall F\in $ ${}_{-1} E_{p,\rho } \left(D^{-} \right)$ there exists a unique expansion 
\[F\left(z\right)=\sum _{n=1}^{\infty }F_{n} F_{p,n}^{-}  \left(z\right)\, ,\, \, \, z\in D^{-} .\] 

The theorem is proved. \end{proof}

\section{Basisness of the double system of $p$-Faber polynomials  with complex coefficients }

Let $D^{+}$ be a bounded, simply connected region, whose boundary $\Gamma $ belongs to class $LR$. Consider the double system of $p$-Faber polynomials
\begin{equation} \label{GrindEQ__6_1_} 
\left\{A\left(\xi \right)F_{p,n}^{+} \left(\xi \right)\, ;\, B\left(\xi \right)F_{p,k}^{-} \left(\xi \right)\, \right\}_{n\ge 0\, ,\, k\ge 1} ,       
\end{equation} 
with complex-valued coefficients $A\left(\xi \right)$ and $B\left(\xi \right)$.  Let us prove the following 

\begin{thm}\label{T6.1} Let $\Gamma $ be a curve of the class $LR$, $0\in int \Gamma$, the functions $A\left(\xi \right)$, $B\left(\xi \right)$ satisfy i)-iii) and $\rho :\Gamma \to R_{+} $ be a weight of the form \eqref{GrindEQ__4_3_}. Let the weights $\rho _{\pm } \left(\cdot \right)$ are defined as in \eqref{GrindEQ__5_0_} on $\partial \omega $. If the weights $\rho \left(\cdot \right)$ and $\rho _{\pm } \left(\cdot \right)$  satisfy $\alpha $) and \eqref{GrindEQ__4_5_}, then the double system of generalized $p$-Faber polynomials \eqref{GrindEQ__6_1_} forms a basis in $L_{p;\rho } \left(\Gamma \right)$, $1<p<+\infty $.
\end{thm}

\begin{proof} Take $f\in L_{p;\rho } \left(\Gamma\right)$  and consider the following nonhomogeneous Riemann problem 
\begin{equation} \label{GrindEQ__6_2_} 
A\left(\xi \right)F^{+} \left(\xi \right)+B\left(\xi \right)F^{-} \left(\xi \right)=f\left(\xi \right)\, ,\, \, \, \xi \in  \Gamma
\end{equation} 
in $E_{p;\rho } \left(D\right)\times _{-1} E_{p;\rho } \left(D^{-} \right)$.   We assume that the coefficients $A\left(\cdot \right)$, $B\left(\cdot \right)$ of the problem \eqref{GrindEQ__6_2_} and the weight $\rho \left(\cdot \right)$ satisfy the conditions of the Theorem \ref{T4.2}. Then, as it follows from Corollary \ref{C4.4}, nonhomogeneous problem \eqref{GrindEQ__6_2_} has a unique solution in \textit{$E_{p,\rho } \left(D^{+} \right)\times $ ${}_{-1} E_{p,\rho } \left(D^{-} \right)$  }for $\forall f\in L_{p;\rho } \left(\Gamma \right)$, which is in the form 
\begin{equation} \label{GrindEQ__6_3_} 
F\left(z\right)=\frac{Z\left(z\right)}{2\pi i} \int _{\Gamma }\frac{f\left(\xi \right)}{Z^{+} \left(\xi \right)} \frac{d\xi }{\xi -z}  ,          
\end{equation} 
where $Z\left(\cdot \right)$ is the canonical solution of the corresponding homogeneous problem.

According to the conditions of the theorem we have  
\begin{equation} \label{GrindEQ__6_5_} 
\rho _{\pm } \left(\cdot \right)\in A_{p} \left(\partial \omega \right)\, ;\, \, \rho \left(\cdot \right)\in A_{p} \left(\Gamma \right)\, ,\, \, 1<p<+\infty .         
\end{equation} 
Then by Theorem \ref{T5.2} the systems of generalized $p$-Faber polynomials  $\left\{F_{p,n}^{+} \right\}_{n\in \mathbb Z_{+} } $ and $\left\{F_{p,n}^{-} \right\}_{n\in \mathbb N} $ form  basis in weighted Smirnov spaces $E_{p,\rho } \left(D^{+} \right)$ and ${}_{-1} E_{p,\rho } \left(D^{-} \right)$, respectively.

Decompose the functions $F^{+} \left(\xi \right)$ and $F^{-} \left(\xi \right)$ in terms of the systems $\left\{F_{p,n}^{+} \right\}_{n\in \mathbb Z_{+} } $  and $\left\{F_{p,n}^{-} \right\}_{n\in \mathbb N} $, respectively
\begin{equation} \label{GrindEQ__6_6_} 
F^{+} \left(\xi \right)=\sum _{n=0}^{\infty }F_{n}^{+} F_{p,n}^{+} \left(\xi \right) , F^{-} \left(\xi \right)=\sum _{n=1}^{\infty }F_{n}^{-} F_{p,n}^{-} \left(\xi \right) ,        
\end{equation} 
where both series converge in $L_{p;\rho } \left(\Gamma\right)$. Set
\[S_{m_{1} ,m_{2} } \left(\xi \right)=A\left(\xi \right)\sum _{n=0}^{m_{1} }F_{n}^{+} F_{p,n}^{+} \left(\xi \right) +B\left(\xi \right)\sum _{n=1}^{m_{2} }F_{n}^{-} F_{p,n}^{-} \left(\xi \right) .\] 
We have
\[\left|f\left(\xi \right)-S_{m_{1} ,m_{2} } \left(\xi \right)\right|_{L_{p;\rho } \left(\Gamma\right)} =\left\| A\left(\xi \right)F^{+} \left(\xi \right)+B\left(\xi \right)F^{-} \left(\xi \right)-S_{m_{1} ,m_{2} } \left(\xi \right)\right\| _{L_{p;\rho } \left(\Gamma\right)} \le \] 
\[\le  \left\| A\left(\xi \right)\right\| _{\infty } \left\| F^{+} \left(\xi \right)-\sum _{n=0}^{m_{1} }F_{n}^{+} F_{p,n}^{+} \left(\xi \right) \right\| _{L_{p;\rho } \left(\Gamma\right)} +\] 
\[+\left\| B\left(\xi \right)\right\| _{\infty } \left\| F^{-} \left(\xi \right)-\sum _{n=1}^{m_{1} }F_{n}^{-} F_{p,n}^{-} \left(\xi \right) \right\| _{L_{p;\rho } \left(\Gamma\right)} \to 0, \] 
as $m_{1} ,\, m_{2} \to \infty $.

Hence,  $\forall f\in L_{p;\rho } \left(\Gamma\right)$ can be expanded as in \eqref{GrindEQ__6_1_}  in $L_{p;\rho } \left(\Gamma\right)$. Now, show that this expansion is unique. Consider the coefficients $\left\{F_{n}^{+} \right\}_{n\in \mathbb Z_{+} } $ of the function $F^{+} \left(\xi \right)$. They are uniquely determined by the function $F^{+} \left(\xi \right)$, i.e. they can be treated as linear functionals of $F^{+} \, :\, F_{n}^{+} =F_{n}^{+} \left(F^{+} \right)\, ,\, \forall n\in \mathbb Z_{+} $. It is clear that 
\begin{equation} \label{GrindEQ__6_7_} 
\left|F_{n}^{+} \right|\le \left\| F_{n}^{+} \right\| \, \left\| F^{+} \right\| _{L_{p;\rho } \left(\Gamma\right)},       
\end{equation} 
where $\left\| F_{n}^{+} \right\| $ are the norms of $F_{n}^{+} :\, E_{p;\rho } \left(D^{+} \right)\to C$. 

Using the Sokhotski-Plemelj formula, from \eqref{GrindEQ__6_3_} we get 
\[F^{\pm } =\pm \frac{1}{2} \frac{f\left(\tau \right)}{A\left(\tau \right)} +\left[S_{\Gamma}^{\pm } f\right]\left(\tau \right), \] 
where    
\[\left[S_{\Gamma}^{\pm } f\right]\left(\tau \right)=\frac{Z^{\pm } \left(\tau \right)}{2\pi i} \int _{\Gamma}\frac{f\left(\xi \right)}{Z^{+} \left(\xi \right)A\left(\xi \right)} \, \frac{d\xi }{\xi -\tau }  \, \, ,\, \, \tau \in \Gamma.\] 
Thus 
\[\left\| F^{+} \right\| _{L_{p;\rho } \left(\Gamma\right)} \le \frac{1}{2} \left\| A^{-1} \right\| _{\infty } \left\| f\right\| _{L_{p;\rho } \left(\Gamma\right)} +\left\| S_{}^{+} f\right\| _{L_{p;\rho } \left(\Gamma\right)} .\] 
By the results obtained in \cite{41,42,43}, if the curve $ \Gamma$ is of class $LR$, under the condition \eqref{GrindEQ__4_5_}  the operator $S_{\Gamma}^{+} $ is bounded in $L_{p;\rho } \left(\Gamma\right)$. As a result we get 

\noindent 
\[\left\| F^{+} \right\| _{L_{p;\rho } \left(\Gamma \right)} \le c\left\| f\right\| _{L_{p;\rho } \left(\Gamma\right)} \, ,\, \forall f\in L_{p;\rho } \left(\Gamma\right).\] 
Taking into account \eqref{GrindEQ__6_7_} we get 
\begin{equation} \label{GrindEQ__6_8_} 
\left|F_{n}^{+} \left(F^{+} \right)\right|\le c_{n} \left\| f\right\| _{L_{p;\rho } \left(\Gamma \right)} \, ,\, \forall f\in L_{p;\rho } \left(\Gamma\right),     
\end{equation} 
where $c_{n} $ are constants independent of $f$. It is clear that $F_{n}^{+} $ also are linear functions of $f$, since $F^{+} $ is uniquely defined by $f$. Denote by the same $F_{n}^{+} $ the generated functionals. From \eqref{GrindEQ__6_8_} we get
\[\left|F_{n}^{+} \left(f\right)\right|\le c_{n} \left\| f\right\| _{L_{p;\rho } \left(\Gamma\right)} \, ,\, \forall f\in L_{p;\rho } \left(\Gamma \right).\] 
The same reasoning is true also for the coefficients $\left\{F_{n}^{-} \right\}_{n\in \mathbb N} $, i.e. we have 
\[\left|F_{n}^{-} \left(f\right)\right|\le c_{n} \left\| f\right\| _{L_{p;\rho } \left(\Gamma \right)} \, ,\, \forall f\in L_{p;\rho } \left(\Gamma \right).\] 
Hence $\left\{F_{n-1}^{+} \, ;\, \, F_{n}^{-} \right\}_{n\in \mathbb N} $  are linear continuous functional in $L_{p;\rho } \left(\Gamma \right)$, i.e. $\left\{F_{n-1}^{+} \, ;\, \, F_{n}^{-} \right\}\subset $ $\left(L_{p;\rho } \left(\Gamma\right)\right)^{*} $.

Now we take $f\left(\xi \right)=A\left(\xi \right)F_{p,k}^{+} \left(\xi \right)$ for any fixed $k\in \mathbb Z_{+} $ and consider the boundary value problem 
\[A\left(\xi \right)F^{+} \left(\xi \right)+B\left(\xi \right)F^{-} \left(\xi \right)=A\left(\xi \right)F_{p,k}^{+} \left(\xi \right)\, ,\, \, \xi \in \Gamma ,\] 
in $E_{p;\rho } \left(D^{+} \right)\times _{-1} E_{p;\rho } \left(D^{-} \right)$. It is clear that the pair $\left(F^{+} \left(z\right)\, ;\, F^{-} \left(z\right)\right)$:
\[\left\{\begin{array}{l} {F^{+} \left(z\right)\equiv F_{p,k}^{+} \left(z\right)\, ,\, z\in D\, ,} \\ {F^{-} \left(z\right)\equiv 0\, ,\, \, z\in D^{-} \, .} \end{array}\right. \] 
Is a solution of above problem. Compare it with the solution \eqref{GrindEQ__6_6_}, by the uniqueness result we get that 
\[F_{n}^{+} =\delta _{nk} \, ,\, \, \forall n\in \mathbb Z_{+} \, ;\, \, F_{m}^{-} =0\, ,\, \, \forall m\in \mathbb N,\] 
i.e.
\[F_{n}^{+} \left(A\left(\xi \right)F_{p,k}^{+} \left(\xi \right)\right)=\delta _{nk} \, ,\, \, \forall n,k\in \mathbb Z_{+} ;\] 
\[F_{m}^{-} \left(A\left(\xi \right)F_{p,k}^{+} \left(\xi \right)\right)=0\, ,\, \, \forall m\in \mathbb N,\, \forall k\in \mathbb Z_{+} ,\] 
where $\delta _{nk} $ is the Kronecker delta.

 By the same way we can write
\[F_{n}^{+} \left(B\left(\xi \right)F_{p,k}^{-} \left(\xi \right)\right)=0\, ,\, \, \forall n,\in \mathbb Z_{+} \, ,\forall k\in \mathbb N;\] 
\[F_{m}^{-} \left(B\left(\xi \right)F_{p,k}^{-} \left(\xi \right)\right)=\delta _{nk} \, ,\, \, \forall n,k\in \mathbb N.\] 
From these equalities it follows that the system \eqref{GrindEQ__6_1_} is minimal in $L_{p;\rho } \left(\Gamma \right)$, and as a result $\forall f\in L_{p;\rho } \left(\Gamma \right)$  has a unique expansion. Theorem is proved. 
\end{proof}

Consider the following special case

\[F_{p,n} \left(\xi \right)=\left\{\begin{array}{l} {F_{p,n}^{+} \left(\xi \right)\, ,\, n\in \mathbb Z_{+} ,} \\ {\, } \\ {F_{p,n}^{-} \left(\xi \right)\, ,\, \, -n\in \mathbb N\, ,} \end{array}\right. \] 
and consider the system 
\begin{equation} \label{GrindEQ__6_9_} 
\left\{e^{i\alpha \arg \xi \, sign\, n} \, F_{p,n} \left(\xi \right)\right\}_{n\in \mathbb Z},       
\end{equation} 
where $\alpha \in \mathbb R$ is a real parameter. Theorem \ref{T6.1} immediately  implies the following 

\begin{cor}\label{C6.1} Let the curve $\Gamma $ is of class $LR$; $0\in int \Gamma$ and the weights $\rho \left(\cdot \right)$;$\rho _{\pm } \left(\cdot \right)$ satisfy the conditions of Theorem \ref{T6.1}. If $-\frac{1}{2q} <\alpha <\frac{1}{2p} $, then the system \eqref{GrindEQ__6_9_} forms a basis in $L_{p;\rho } \left(\Gamma \right), 1<p<+\infty$.
\end{cor}

 ------------------------------------------------------------------------


\end{document}